\documentclass[11pt, reqno]{amsart}
\usepackage{amsmath}
\usepackage{amscd,amsthm,amsfonts,amsopn,amssymb,verbatim}

\def\vp{\varphi}

\def\be{\begin{equation}}
\def\ee{\end{equation}}

\newtheorem{theorem}{Theorem} [section]

\newtheorem{lemma}[theorem]{Lemma}%[section]

\theoremstyle{remark}

\def\be{\begin{equation}}
\def\ee{\end{equation}}
\def\vp{\varphi}

\begin{document}

\allowdisplaybreaks
\title[]
{On Pleijel's Nodal Domain Theorem}
\author
{J.~Bourgain}
\address
{Institute for Advanced Study, Princeton, NJ 08540}
\email
{bourgain@math.ias.edu}
\thanks{This work was partially supported by NSF grant DMS-1301619}

\begin{abstract}
A slight improvement of Pleijel's estimate on the number of nodal domains is
obtained, exploiting a refinement of the Faber-Krahn inequality and packing
density of discs.
\end{abstract}
\maketitle

\allowdisplaybreaks

\section
{Introduction}

It is proved in \cite {10} that if $\Omega\subset\mathbb R^2$ is a membrane with fixed boundary and \break
$\lambda_1\leq \lambda_2\leq\cdots\leq \lambda_n\leq \cdots$ the Dirichlet spectrum of $\Omega$,
then the number $N=N_n$ of nodal domains of the eigenfunction $\vp =\vp_{\lambda_n}$ satisfies the inequality
\be\label{1}
\underset {n\to\infty}{\hbox{$\lim \sup$}}\  \frac Nn \leq\Big(\frac 2j\Big)^2 =0.691...
\ee
with $j=2.4048...$ being the smallest positive zero of the Bessel function $J_0$.
Note that Pleijel's constant is in fact quite close to the best possible one, because on a rectangle one obtains
 $\frac 2\pi = 0,636\cdots .$
It was in fact conjectured that the constant $\frac 2\pi$ is optimal (cf. \cite{11}).

While it was already pointed out in \cite{11} that Pleijel's inequality is not sharp,
 the purpose of this note is to carry out a refinement of Pleijel's argument that leads to an explicit small 
improvement of \eqref{1}.
Although the gain over \eqref {1} is minuscule and by itself of little interest, we feel that the argument deserves 
some attention and may have further developments. Note also that combining the approach of \cite{11} with the techniques from the present paper,
one gets a similar improvement in Pleijel's constant as stated in the proposition below, for piecewise analytic domains with Neumann boundary condition.

 The underlying idea is very simple and based on two ingredients.

The first is the stability property of the Faber-Krahn inequality and the other is the packing density of discs.

Let us briefly review the reasoning leading to \eqref {1}.

Denote $\Omega_1, \ldots, \Omega_N$ the nodal domains of $\vp=\vp_{\lambda_n}$.
Then, for each $\alpha =1, \ldots, N$, we have that
\be\label{2}
\lambda_n\geq \lambda_1(\Omega_\alpha).
\ee
According to Faber-Krahn inequality,
\be
\label{3}
\lambda_1(\Omega_\alpha)\geq \lambda_1 (\mathcal D_\alpha)= \frac {\pi\lambda_1(\mathcal D)}{|\Omega_\alpha|},
\ee
where $\mathcal D$ is the unit disc and $\mathcal D_\alpha$ a disc of same area as $\Omega_\alpha$.
Note that $\lambda_1(\Omega_\alpha)> \lambda_1(\mathcal D_\alpha)$ unless $\Omega_\alpha$ is itself a disc, which is the stability property 
brought up earlier, with explicit analytical formulations in terms of Fraenkel asymmetry of Hausdorff distance.
We will rely on Lemma 2.1 stated below, which is an extension of a result from \cite{4} for domains that are not
necessarily simply connected.

Since $\lambda_1(\mathcal D)=j^2, j =$ smallest positive zero of $J_0$, \eqref {2}, \eqref {3} imply
\be\label{4}
\lambda_n|\Omega_\alpha|\geq j^2\pi.
\ee
Summing \eqref{4} over $\alpha =1, \ldots, N$ and observing that obviously
\be
\label{5}
\sum^N_{\alpha=1} |\Omega_\alpha|= |\Omega|
\ee
gives
\be\label{6}
\lambda_n|\Omega|\geq j^2\pi N.
\ee
According to Weyl's asymptotic law,

\be\label{7}
\frac{\lambda_n}n\to \frac {4\pi}{|\Omega|} \text { for $n\to\infty$}.
\ee

Combining (6), (7), we deduce \eqref{1}.

Recall that (3) is only saturated if $\Omega_\alpha$s are discs.
Since, roughly speaking, one would then expect most of these discs to be of size $O\big(\frac 1{\sqrt n}\big)$, bounds on the packing density of a family of discs
of comparable sizes (cf. Lemma 2.2) will improve upon the use of (5).

The results presented in Section 2 will permit to make the improvements discussed above quantitative.
\bigskip

\noindent{\bf Proposition.}
\be\label{8}
\underset {n\to\infty} {\hbox{$\lim\sup$}} \ \frac Nn<\Big(\frac 2j\Big)^2 -(3.\,  \ldots)10^{-9}.
\ee
\bigskip

The argument is given in Section 3.
Needless to say, one can likely do better, for instance by improving the constants in Lemma 2.1, which are certainly not
optimal.

Some of the author's original motivation for looking into Courant's nodal domain theorem lies in the predictions made
by Bogomolny and Schmit \cite{2}
on the mean value of $N =$ number of connected nodal domains, in the random wave function model (also conjectured to apply for quantum eigenfunctions of chaotic
manifolds) and based on percolation theory.
It is indeed interesting to note that in the prediction [2, inequality (5)], the ratio
\be \label{9}
\lim_{E\to \infty} \ \frac {\bar N(E)}{\bar n(E)} \approx 0.0624\ldots
\ee
is quite small.
From a rigorous point of view, it is not unreasonable to expect further improvements related to the application of the Faber-Krahn inequality for random models 
(the work of Nazarov and Sodin \cite {9}, while establishing a limiting distribution for $\frac Nn$ when $n\to\infty$ for random spherical harmonics (or,
equivalently, the random wave model),  apparently does not shed light on the actual value of the Bogomolny-Schmit constant.

The author was very recently informed of work of M.~Krishnapur \cite{7} on bounding the number of nodal domains of random plane waves, within the
Bogomolny-Schmit prediction by a factor fo 3,604... and also a numerical study by M.~Nastasescu \cite {8}.

\section
{Preliminaries}

The first ingredient is a refinement of the Faber-Krahn inequality which expresses the stability in a quantitative
form.
In 2D, there are several such results in the literature, in particular the paper [4] of Hansen and Nadirashvili and [0] 
due to Bhattacharya.
Theorem 4.1 in [4], valid for simply connected domains, has explicitly stated constants and our next Lemma 2.1 is an 
extension of it to the multiply connected case.

\begin{lemma}\label{Lemma1}
Let $G$ be a nonempty bounded domain in $\mathbb R^2$, which is finitely connected and denote $0<\gamma\leq \infty$ the
minimum area of a component of $\mathbb R^2\backslash G$. Then
\be\label{10}
\lambda_1(G) \geq \lambda_1 (G_0)\Big[1+\frac 1{250} \min\Big(1- \frac{r_i(G)}{r_0(G)}, 2\frac {\sqrt\gamma}{r_0(G)}\Big)^3\Big] 
\ee
with $G_0$ the disc of area $|G_0|=|G|$, $r_0(G)$ the radius of $G_0$ and $r_i(G)$ the inradius of $G$.
\end{lemma}

\begin{proof}
Assume $|G|=1$, as we may by rescaling.

We follow \cite{4}, \S4.
Let $u>0$ on $G$, $u=0$ on $\partial G$ satisfy $\int_G u^2=1$ and $\Delta u+ \lambda_1(G) u=0$ and let
$u_0:G_0\to\mathbb R$ be the spherical rearrangement of $u$, i.e. $\vp(t)=|[u>t]|=|[u_0>t]|$, $0\leq t\leq
T=\max u(G)$.
Then $-\vp'(t) =\int_{\{u=t\}} |\partial_n u|^{-1} d\sigma$ and by H\"older's inequality
\be\label{11}
\lambda_1(G)=\int_G|\nabla u|^2 =\int^T_0 \Big(\int_{\{u=t\}}|\partial_n u|d\sigma\Big)
dt \geq \int_0^T \frac {\sigma^2(u=t)}{-\vp'(t)} dt.
\ee
Let $t>0$.
\medskip

\noindent
{\bf Claim.}
{\it Denote $\Omega=[u>t]$. Then
\be\label{12}
\sigma^2(u=t)\geq \sigma^2 (u_0=t)+4\pi\min \big(\big(r_0(\Omega)-              r_i(\Omega)\big)^2, \gamma\big).
\ee}

Assuming (12), define
\be\label{13}
\delta_1 =\min \Big(\frac 12 \Big( 1-\frac {r_i(G)}{r_0(G)}\Big), \sqrt{\pi\gamma}\Big)
\ee
and take $0<s<T$ such that
\be\label{14}
|[u>s]|=1-\delta_1 =(1-\delta_1)|G|.
\ee
For $0\leq t\leq s$, it follows that
$$
r_0 ([u>t])\geq r_0 ([u>s])\geq (1-\delta_1)^{\frac 12} r_0(G)> (1-\delta_1)    r_0(G)
$$
and since $r_i([u>t])\leq r_i (G)$, by (13)
$$
r_0([u>t])-r_i([u>t])\geq (1-\delta_1) r_0(G) -r_i(G)\geq \frac 12 \big(r_0(G)- r_i(G)\big).
$$
Thus, by (12), for $0<t\leq s$
\be\label{15}
\sigma^2(u=t)\geq \sigma^2(u_0=t)+\pi\min \big((r_0-r_i)^2, 4\gamma\big).
\ee
Substituting \eqref{15} in \eqref{11} gives
\begin{align}\label{16}
\eqref{11}&\geq \int_0^T\frac{\sigma^2 (u_0=t)}{-\vp'(t)} dt +\pi\min\big((r_0-r_i)^2, 4\gamma\big)\ \int_0^s \frac
1{-\vp'(t)} dt\nonumber\\
&=\int_{G_0}|\nabla u_0|^2 +\pi\min \big((r_0-r_i)^2, 4\gamma\big) \frac {s^2}{\delta_1}\nonumber\\
&\geq \lambda_1 (G_0)+4\delta_1 s^2
\end{align}
since $\int^s_0\frac 1{-\vp'} \geq \frac {s^2}{\int_0^s-\vp'} =\frac {s^2}{\delta_1}$ and (13).
\medskip

As in [4], we distinguish two cases.
\medskip

\noindent
{\bf Case 1.} $s\geq \frac 25 \delta_1$.

Then from \eqref{16} and since $\lambda_1(G_0)<20$
\be\label{17}
\lambda_1 (G)\geq \Big(1+\frac 4{125}+\delta^3_1\Big)\lambda_1 (G_0).
\ee

\noindent
{\bf Case 2.} $s<\frac 25 \delta_1$.

Write
$$
\begin{aligned}
\int_{[u>s]} (u-s)^2 &=1-\int_{[u\leq s]} u^2 -2s\int_{[u>s]} u+s^2|[u>s]|\\
&\geq 1-s^2 -2s> 1-\frac {22}{25} \delta_1.
\end{aligned}
$$
Also, using the Faber--Krahn inequality
$$
\frac{\int_{[u>s]}|\nabla u|^2}{\int_{[u>s]} (u-s)^2} \geq\lambda_1            ([u>s])\geq \lambda_1 ([u_0>s]) =
\frac {\lambda_1(G_0)}{|[u>s]|} =\frac {\lambda_1(G_0)}{1-\delta_1}
$$
implying
\be\label{18}
\lambda_1 (G)\geq \int_{[u>s]} |\nabla u|^2\gtrsim \frac {1-\frac {22}{25}     \delta_1}{1-\delta_1} \lambda_1
(G_0)>\Big(1+\frac 3{25}\delta_1\Big) \lambda_1(G_0).
\ee
Thus \eqref{17} holds in both cases, implying (10).
\medskip

\noindent
{\bf Proof of the claim}

Denote $\{\Omega_\alpha\}_{\alpha\geq 0}$ the connected components of $\Omega$ and assume $|\Omega_0|\geq
|\Omega_\alpha|$.
We distinguish two cases.

\medskip

\noindent
{\bf Case 1.} $\Omega_0$ is simply connected.

Applying Bonnesen's inequality and the isoperimetric inequality gives
\begin{align}\label{19}
\sigma^2(\partial \Omega_0)&\geq 4\pi |\Omega_0|+ 4\pi \big(r_0(\Omega_0)- r_i(\Omega_0)\big)^2
\\
\label{20}\sigma^2(\partial\Omega_\alpha) & \geq 4\pi|\Omega_\alpha| \ \text { for }     \alpha> 0.
\end{align}
Also
$$
\begin{aligned}
\sigma^2(\partial\Omega) &= \Big(\sigma(\partial\Omega_0)+\sum_{\alpha>0}      \sigma(\partial\Omega_\alpha)\Big)^2\\
&\geq \sigma^2(\partial \Omega_0)+\sum_{\alpha >0}
\sigma^2(\partial\Omega_\alpha)+2\sigma(\partial\Omega_0)
\sum_{\alpha>0} \sigma(\partial\Omega_\alpha)
\end{aligned}
$$
and by \eqref{19}, \eqref{20}
\be\label{21}
\geq 4\pi|\Omega|+ 4\pi \big(r_0(\Omega_0)-r_i(\Omega_0)\big)^2 +  8\pi\sum_{\alpha>0} |\Omega_\alpha|.
\ee
On the other hand
$$
\begin{aligned}
\big(r_0(\Omega) -r_i(\Omega)\big)^2&\leq \big(r_0(\Omega)-
r_i(\Omega_0)\big)^2 \leq \frac {|\Omega|}\pi + r_i(\Omega_0)^2-
2\frac {|\Omega_0|^{\frac 12}}{\sqrt\pi} r_i(\Omega_0)\\
&= \big(r_0(\Omega_0) - r_i(\Omega_0)\big)^2 +\frac 1\pi \sum_{\alpha
>0} |\Omega_\alpha|
\end{aligned}
$$
and therefore \eqref{21} $\geq 4\pi|\Omega|+4\pi \big(r_0(\Omega) -
r_i(\Omega)\big)^2$.

\medskip
                                                             
\noindent
{\bf Case 2.} $\Omega_0$ is not simply connected.

Then there are Jordan domains $D_1\subset D$ such that $\Omega_0\subset D\backslash D_1$, $\partial    D\cup\partial
D_1\subset
\partial\Omega_0$.
Thus $u=t>0$ on $\partial D\cup\partial D_1 \subset G$.
We claim that $D_1\backslash G\not=\phi$.
Otherwise $D_1\subset G$ and $[u<t]\cap D_1 \not=\phi$ (if $u\geq t$ on $D_1$ it would follow from the eigenvalue
equation that
$u$ vanishes at infinite order on $\partial D_1$).
But since $u=t$ on $\partial D_1$ and $u$ is the lowest eigenvalue eigenfunction of $G$, \break
$[u<t]\cap D_1=\phi$.
Hence $D_1\backslash G\not= \phi$ and $D_1$ contains a component of $\mathbb R^2\backslash G$.
Thus
\be\label{22}
|D_1|\geq \gamma
\ee
By the isoperimetric inequality again
$$
\begin{aligned}
\sigma^2(\partial\Omega_0)&\geq \sigma^2(\partial D)\geq 4\pi |D|\geq 4\pi |\Omega_0|+ 4\pi\gamma\\
\sigma^2(\partial\Omega_\alpha) &\geq 4\pi|\Omega_\alpha| \text { for } \alpha>0.
\end{aligned}
$$
Hence
$$
\sigma^2(\partial\Omega)\geq \sigma^2(\partial\Omega_0)+\sum_{\alpha>0}
\sigma^2(\partial\Omega_\alpha)\geq
4\pi|\Omega|+4\pi\gamma.
$$
This proves the claim and Lemma 2.1.
\end{proof}

The second ingredient concerns packing density of discs and appears in the paper Blind \cite{1}.

\begin{lemma}\label{Lemma2.2}
The packing density of a family of discs of radii $r_i$ such that
\be\label{23}
\inf \frac {r_i}{r_j}\geq p=0.74299\ldots
\ee
is at most $\frac\pi{\sqrt {12}}$.
\end{lemma}

Note that $\frac \pi{\sqrt {12}}$ is the optimal packing density of congruent discs in the plane.

It turns out that only the precise form of \eqref{10} determines the improvement  obtained in inequality \eqref{8}.

\section
{Proof of the Proposition}

Following the argument and notation from Section 1, an application of \eqref {10} with $G=\Omega_\alpha$ gives
\be
\label{24}
\lambda_n|\Omega_\alpha|\geq j^2\pi \Big[ 1+\frac 1{250} 
\min \Big(1-\frac {r_i(\Omega_\alpha)}{r_0(\Omega_\alpha)}, 2 \frac {\sqrt \gamma}{r_0(\Omega_\alpha)}\Big)^3\Big]
\ee
instead of \eqref{4}, where $\gamma$ stands for the minimum area of a component of $\mathbb
R^2\backslash\Omega_\alpha$.

We introduce parameters $\delta>0$ and $\rho_+, \rho_- =p\rho_+$ ($p$ given by \eqref {22}), to be specified, and
partition the $\Omega_\alpha'$ in the following classes:

\begin{itemize}
\item [(I)] \ $r_i(\Omega_\alpha)\leq (1-\delta) r_0(\Omega_\alpha), r_0(\Omega_\alpha)\leq \rho_+$
\item [(II)] \ $r_0(\Omega_\alpha)>\rho_+$
\item [(III)] \ $r_i(\Omega_\alpha)> (1-\delta) r_0 (\Omega_\alpha) $ and $\rho_-\leq r_i(\Omega_\alpha)\leq\rho_+$
\item [(IV)] \ $ r_i(\Omega_\alpha)> (1-\delta)r_0(\Omega_\alpha)$ and $r_i(\Omega_\alpha)<\rho_-$
\end{itemize}

Let $N_I, N_{II} , N_{III}, N_{IV}$ be the respective number of those domains.

We first define $\rho_-$ in order to ensure that $N_{IV}=0$.

Since by \eqref{4}
\be\label{25}
\lambda_n \ r_0(\Omega_\alpha)^2 \geq j^2
\ee
and $r_0(\Omega_\alpha)< (1-\delta)^{-1}\rho_-$ for $\alpha $ class (IV), we ensure $N_{(IV)}=0$ by taking
\be\label{26}
\rho_- = j(1-\delta)\lambda_n^{-\frac 12 }\text { and } \rho_+ =\frac 1p j(1-\delta) \lambda_n^{-\frac 12}.
\ee

Note that a component of $\mathbb R^2\backslash \Omega_\alpha$ necessarily contains another nodal domain
$\Omega_\beta$ (unless of infinite area) and hence, by the Faber-Krahn inequality, is at least of area
$\gamma\geq \frac {j^2\pi}{\lambda_n}$.
Hence $2\frac {\sqrt \gamma}{r_0(\Omega_\alpha)} \geq \delta$ if $\Omega_\alpha$ is class (I) and therefore
$$
\lambda_n|\Omega_\alpha|\geq j^2\pi \Big(1+\frac 1{250} \delta^3\Big).
$$
Hence
\be\label{27}
\lambda_n \sum_{(I)} |\Omega_\alpha| \geq j^2\pi \Big(1+\frac 1{250} \delta^3\Big)N_I.
\ee
Clearly
\be\label{28}
\sum_{(II)} |\Omega_\alpha|\geq \pi\rho^2_+ N_{II}
\ee
and by \eqref{26}
\be\label{29}
\lambda_n\sum_{(II)} |\Omega_\alpha|\geq\pi \Big(\frac {j(1-\delta)}{p}\Big)^2 N_{II}.
\ee

Also, by \eqref{4}
\be\label{30}
\lambda_n\sum_{(III)} |\Omega_\alpha| \geq j^2\pi N_{III}.
\ee
Adding up \eqref{27}, \eqref {29}, \eqref {30}, we obtain, since $N=N_I+N_{II}+N_{III}$
\be\label {31}
\lambda_n|\Omega|\geq j^2 \pi\Big\{ N+\frac 1{250} \delta^3N_I +\Big(\Big(\frac {1-\delta}{p}\Big)^2 -1\Big)
N_{II}\Big\}.
\ee
Returning to \eqref{30}, write

\be\label{32}
\lambda_n (1-\delta)^{-2} \sum_{(III)} r_i(\Omega_\alpha)^2 \geq j^2  N_{III}.
\ee
We exploit Lemma \ref{Lemma2.2}. Each $\Omega_\alpha$ of class $(III)$ contains a disc $\mathcal D_\alpha$ of radius
$r_i(\Omega_\alpha)$ subject to the constraints $p\rho_+\leq r_i(\Omega_\alpha)\leq \rho_+$.
Considering the family $\{{\mathcal D}_\alpha;  \alpha \text { class (III)} \}$ of discs in $\Omega$, it follows from
Lemma \ref{Lemma2.2} that   in the limit for $n\to\infty$
\be\label{33}
\pi\sum_{(III)} r_i(\Omega_\alpha)^2 < \Big(\frac\pi{\sqrt {12}}+ o(1)\Big) \,|\Omega|.
\ee
One comment should be made at this point.
It is known that for $n\to\infty$ the number of corresponding nodal domains does not go necessarily to infinity
\cite{3}.
In particular, the class (III) could be empty, which of course does not violate the previous statement.

Combined with \eqref{32}, it follows that
\be\label{34}
\big(1+o(1)\big)\lambda_n(1-\delta)^{-2} \frac {|\Omega|}{\sqrt{12}} > j^2 N_{III} =j^2(N-N_I-N_{II}).
\ee
Restrict $0<\delta<1$ such that
\be\label{35}
\frac {\delta^3}{250} \leq \Big(\frac {1-\delta} p\Big)^2 -1.
\ee
If \eqref{35}, \eqref{31} implies
\be\label{36}
\lambda_n|\Omega| \geq j^2\pi \Big(N+\frac {\delta^3}{250} (N_I+N_{II})\Big).
\ee
Substituting in \eqref{34} gives the inequality
$$
\big(1+o(1)\big) \lambda_n(1-\delta)^{-2} \frac {|\Omega|}{\sqrt{12}}\geq j^2 (1+250 \delta^{-3})N-250 \delta^{-3} \,
\frac   {\lambda_n|\Omega|}{\pi} 
$$
and
\be\label{37}
N < \frac {(1+o(1))\lambda_n|\Omega|}{j^2(1+250 \delta^{-3})} \ \Big[ \frac 1{\sqrt{12}} (1-\delta)^{-2} +\frac
{250}\pi \delta ^{-3}\Big].
\ee
Invoking again the Weyl asymptotic \eqref{7}, \eqref{37} implies
$$
\overline{\lim_{n\to\infty}} \ \frac Nn \leq \Big(\frac 2j\Big)^2 (250+\delta^3)^{-1} \Big[250 +\frac {\pi}{\sqrt{12}}
\ \frac {\delta^3}{(1-   \delta)^2}\Big]
$$
and
\be\label{38}
\Big(\frac 2j\Big)^2 \underset{n\to\infty}{-\overline\lim} \  \frac Nn \geq \Big(\frac
2j\Big)^2 \frac {\delta^3}{250+\delta^3} \Big(1-       \frac\pi{\sqrt{12}} \ \frac
1{(1-\delta)^2}\Big).
\ee

It remains to optimize the r.h.s. of \eqref{38} in $\delta$,
which was performed using Mathematica software.
One verifies that since certainly $(1-\delta)^2 > \frac \pi{\sqrt{12}}$, inequality \eqref{35} is automatically
fulfilled.

The conclusion is as stated in \eqref{8}.
\medskip

{\bf Remark.}
Our result is also of some relevance to the theory of spectral minimal partitions (see \cite {5} for a survey).
In particular, it gives a slight improvement in the inequality $\mathcal L_k(\Omega)\geq k\frac {\pi j^2}{|\Omega|} $
(see [6, Proposition 6.1], where $\{\mathcal L_k\}_{k\geq 1}$ is the spectral minimal partition sequence.
\hfill $\square$
\medskip 

\noindent
{\bf Acknowledgements} 

The author is grateful to E.~Bourgain-Chang for carrying out the optimization problem at the
end
of the paper and an anonymous referee for various comments. 

Most importantly, he wishes to thank H.~Donnelly for pointing out an omission in an earlier version of this paper, that
led to Lemma 2.1 in this version and for several email correspondences.

This work was carried out at UC Berkeley and the author thanks the mathematics department for its hospitality.

\noindent
{\bf Funding}

Partially funded by NSF grants DMS-1301619 and DMS-0835373 (to J.B.).

\end{document}